\documentclass[11pt,letterpaper,reqno]{amsart}

\usepackage{amssymb}
\usepackage{amsmath}
\usepackage{amsthm}
\usepackage{amsfonts}
\usepackage{bbm}
\usepackage{graphicx}
\usepackage[T1]{fontenc}
\usepackage{doi}
\usepackage{bookmark}
\usepackage{hyperref}
\hypersetup{pdfstartview={FitH}}

\newtheorem{thm}{Theorem}[section]
\newtheorem{lem}[thm]{Lemma}
\newtheorem{prop}[thm]{Proposition}
\newtheorem{cor}[thm]{Corollary}
\newtheorem{claim}{Claim}
\theoremstyle{definition}
\newtheorem{remark}[thm]{Remark}
\newtheorem{problem}[thm]{Problem}
\newtheorem{conj}[thm]{Conjecture}
\newtheorem{defn}[thm]{Definition}
\newtheorem{exm}[thm]{Example}
\numberwithin{equation}{section}

\newcommand{\seqnum}[1]{\href{https://oeis.org/#1}{\rm \underline{#1}}}

\begin{document}

\title[On a conjecture of Erd\H{o}s and Graham about the Sylvester's sequence]{On a conjecture of Erd\H{o}s and Graham about the Sylvester's sequence}

\author[Zheng Li]{Zheng Li}
\author[Quanyu Tang]{Quanyu Tang}

\address{School of Mathematics and Statistics, Xi'an Jiaotong University, Xi'an 710049, P. R. China}
\email{math.zhengli@gmail.com}

\address{School of Mathematics and Statistics, Xi'an Jiaotong University, Xi'an 710049, P. R. China}
\email{tang\_quanyu@163.com}



\begin{abstract}
Let $\{u_n\}_{n=1}^{\infty}$ be the Sylvester's sequence (sequence A000058 in the OEIS), and let $ a_1 < a_2 < \cdots $ be any other positive integer sequence satisfying $ \sum_{i=1}^\infty \frac{1}{a_i} = 1 $. In this paper, we solve a conjecture of Erd\H{o}s and Graham, which asks whether $$ \liminf_{n\to\infty} a_n^{\frac{1}{2^n}} < \lim_{n\to\infty} u_n^{\frac{1}{2^n}} = c_0 = 1.264085\ldots. $$ We prove this conjecture using a constructive approach. Furthermore, assuming that the unproven claim of Erd\H{o}s and Graham that "all rationals have eventually greedy best Egyptian underapproximations" holds, we establish a generalization of this conjecture using a non-constructive approach.
\end{abstract}

\maketitle

\tableofcontents

\section{Introduction and Main Results}\label{sec1}

The ancient Egyptians developed a distinctive approach to working with fractions, setting their system apart from those of other civilizations. They introduced specific symbols for \textit{unit fractions}, which are reciprocals of integers, such as \( \frac{1}{2}, \frac{1}{3}, \frac{1}{7} \), and so on. However, they lacked symbols for fractions with numerators greater than 1, with the sole exception of \( \frac{2}{3} \). Instead, they expressed such fractions as sums of unit fractions. Although a fraction like \( \frac{3}{53} \) could be represented trivially as \( \frac{1}{53} + \frac{1}{53} + \frac{1}{53} \), the Egyptians preferred representations without repeated terms and sought more concise expressions. Consequently, the term \textit{Egyptian fraction} refers to a representation of a fraction as a sum of \textit{distinct} unit fractions. For example, the number 1 can be expressed as \( \frac{1}{2} + \frac{1}{3} + \frac{1}{6} \).

Egyptian fractions give rise to many problems that are simple to state but challenging to solve, making them a long-standing area of interest for the late Paul Erd\H{o}s \cite[\S 4]{EG80}. Comprehensive survey articles on this topic have been authored by Graham \cite{Gra13} and by Bloom and Elsholtz \cite{BE22}.

In this paper, we primarily focus on representing a number using an Egyptian fraction with an infinite number of terms. This necessitates introducing concepts such as \textit{Egyptian underapproximation} and related notions. For a more detailed discussion, readers are referred to Kova\v{c} \cite{Kovac2025}.

For a positive integer $n$, we say that a  positive integer sequence $\{x_i\}_{i=1}^n$ is an \emph{$n$-term Egyptian underapproximation}\footnote{In other literature, such as Kova\v{c} \cite{Kovac2025} and Nathanson \cite{Nat23}, the definition of Egyptian underapproximation varies: Kova\v{c} defines it as the sum \( \sum_{i=1}^{n} 1/x_i \), while Nathanson defines it as the \( n \)-tuple \( (x_1, x_2, \ldots, x_n) \).} of a rational number $0 < \lambda \leq 1$ if
\begin{equation}\label{eq:under}
\sum_{i=1}^{n} \frac{1}{x_i} < \lambda ,
\end{equation}
where $x_1\leq x_2 \leq \cdots \leq x_n$. We also say a sequence \( \{x_i\}_{i=1}^{\infty} \) is an \emph{infinite Egyptian underapproximation} of \( 0 < \lambda \leq 1 \) if, for every positive integer \( m \), the finite subsequence \( \{x_i\}_{i=1}^{m} \) is an \emph{\( m \)-term Egyptian underapproximation} of \( \lambda \).

A practical approach to approximating a number $0 < \lambda \leq 1$ from below using sums of distinct unit fractions is the greedy algorithm. We define the \emph{greedy \( n \)-term Egyptian underapproximation} of $0 < \lambda \leq 1$ as the sequence $\{x_i\}_{i=1}^n$ determined recursively by \[
x_i := \Big\lfloor \Big(\lambda - \sum_{j=1}^{i-1} \frac{1}{x_j} \Big)^{-1} \Big\rfloor + 1,
\] with the convention that the empty sum \( \sum_{j=1}^{0} \) is taken as \( 0 \). We also define the \emph{infinite greedy Egyptian underapproximation} of \( 0 < \lambda \leq 1 \) as the sequence \( \{x_i\}_{i=1}^{\infty} \) if, for every positive integer \( m \), the finite subsequence \( \{x_i\}_{i=1}^{m} \) is the \emph{greedy \( m \)-term Egyptian underapproximation} of \( \lambda \).

A well-known example is precisely the so-called Sylvester's sequence.
\begin{exm}[Sylvester's Sequence]
Let the sequence \( \{u_n\}_{n=1}^{\infty} \) be the infinite greedy Egyptian underapproximation of 1. That is, \( \{u_n\}_{n=1}^{\infty} \) is defined such that its \( n \)th term \( u_n \) corresponds to the denominator of the \( n \)th term in the following infinite Egyptian fraction representation of 1:
\[
1 = \frac{1}{2} + \frac{1}{3} + \frac{1}{7} + \frac{1}{43} + \frac{1}{1807} + \frac{1}{3263443} + \frac{1}{10650056950807} + \cdots.
\]
We refer to the sequence \( \{u_n\}_{n=1}^{\infty} \) as the \textit{Sylvester’s sequence} \cite{Sylvester1880}, which is defined recursively by the following rule: \[
u_1 = 2, \quad \text{and} \quad u_{n+1} = \prod_{k=1}^{n} u_k + 1, \quad \text{for all } n \geq 1.
\] Sylvester’s sequence is sequence \seqnum{A000058} in the OEIS \cite{OEIS}.
\end{exm}
\begin{remark}
Instead of defining the Sylvester’s sequence \( \{u_n\}_{n=1}^{\infty} \) in terms of all previous terms, \( u_{n+1} \) can be expressed solely in terms of \( u_n \) as follows for \( n \geq 1 \): \[
u_{n+1} = 1 + \prod_{k=1}^{n} u_k = 1 + u_n \prod_{k=1}^{n-1} u_k = 1 + u_n(u_n - 1) = u_n^2 - u_n + 1.
\] Furthermore, it is straightforward to verify that  
\[
\sum_{i=1}^{n-1} \frac{1}{u_i} + \frac{1}{u_n - 1} = 1, \quad \text{for } n \geq 1,
\]
with the convention that the empty sum \( \sum_{k=1}^{0} \) is taken as \( 0 \).
\end{remark}
For additional properties of the Sylvester’s sequence, readers may refer to \cite{chentoufsylvester, chun2011egyptian}. In this paper, we primarily focus on the following conjecture proposed by Erd\H{o}s and Graham in \cite[p.~41]{EG80}:
\begin{conj}\label{conj1}
Let \( u_1 = 2 \) and \(
u_{n+1} = u_n^2 - u_n + 1.
\)\footnote{In \cite{EG80}, this problem is stated with the sequence \( u_1 = 1 \) and \( u_{n+1} = u_n (u_n + 1) \), but this is an obvious typo, since with that choice we do not have \(\sum_{i=1}^{\infty} \frac{1}{u_i} = 1.\) This question with the Sylvester's sequence is the most natural interpretation of what they meant to ask.
} Let \( a_1 < a_2 < \cdots \) be any other positive integer sequence with \[
\sum_{i=1}^\infty \frac{1}{a_i} = 1.
\] Then \[
\liminf_{n\to\infty} a_n^{\frac{1}{2^n}} < \lim_{n\to\infty} u_n^{\frac{1}{2^n}} = c_0\footnote{$c_0$ is called the Vardi constant, whose decimal expansion corresponds to sequence \seqnum{A076393} in the OEIS \cite{OEIS}. See also \cite[p.~444]{Finch2003}.} = 1.264085\ldots.
\]
\end{conj}

This conjecture is listed as Problem \#315 on Thomas Bloom's Erd\H{o}s Problems website \cite{EP}.

Inspired by the Sylvester's sequence, we introduce the following concept:

\begin{defn}\label{def1}
Let \(\{a_n\}_{n=1}^{\infty}\) be a sequence of positive integers. We say that \(\{a_n\}_{n=1}^{\infty}\) is \textit{eventually Sylvester} if there exists a positive integer \(N > 0\) such that for all \(n \geq N\), the sequence satisfies the following recurrence relation:\footnote{Badea \cite{Badea1995} recorded several sufficient conditions for a sequence to be eventually Sylvester, such as \cite[Theorem 3--5]{Badea1995}.}
    \[
    a_{n+1} = a_n^2 - a_n + 1.
    \]
\end{defn}

This paper aims to provide proofs of Conjecture \ref{conj1} and its generalization: one constructive and one non-constructive.

\subsection{Constructive Method} To resolve Conjecture \ref{conj1} using a constructive method, we first need to establish the following result:
\begin{thm}\label{mainthm}
Let \( \{u_n\}_{n=1}^{\infty} \) be the Sylvester’s sequence, and let \( N \geq 2 \) be a positive integer. Let \( \{a_n\}_{n=1}^{\infty} \) be any eventually Sylvester sequence of positive real numbers satisfying the recurrence relation  
\[
a_{n+1} = a_n^2 - a_n + 1, \quad \text{for } n \geq N.
\]
Suppose that \( \{a_n\}_{n=1}^{\infty} \) also satisfies the following two conditions:  
\begin{itemize}
    \item[(*)] \(\sum_{i=1}^\infty \frac{1}{a_i} = 1\);
    \item[(**)] \(\sum_{i=1}^{N-1} \frac{1}{a_i} < \sum_{i=1}^{N-1} \frac{1}{u_i}\).
\end{itemize}
Then
\[
\liminf_{n\to\infty} a_n^{\frac{1}{2^n}} = \lim_{n\to\infty} a_n^{\frac{1}{2^n}}< \lim_{n\to\infty} u_n^{\frac{1}{2^n}}.
\]
\end{thm}

To fully prove Conjecture \ref{conj1}, it remains to establish the following constructive result.

\begin{thm}\label{mainthm0}
Let \( \{u_n\}_{n=1}^{\infty} \) be the Sylvester’s sequence, and let \( N \geq 2 \) be a positive integer. Let \( a_1 < a_2 < \cdots \) be any other sequence of positive integers satisfying  
\[
\sum_{i=1}^\infty \frac{1}{a_i} = 1.
\]
Then, we can construct an eventually Sylvester sequence \( \{c_n\}_{n=1}^{\infty} \) of positive real numbers that satisfies the recurrence relation  
\[
c_{n+1} = c_n^2 - c_n + 1, \quad \text{for } n \geq N,
\]
and the following two conditions:
\begin{itemize}
    \item[(*)] \(\sum_{i=1}^\infty \frac{1}{c_i} = 1\);
    \item[(**)] \(\sum_{i=1}^{N-1} \frac{1}{c_i} < \sum_{i=1}^{N-1} \frac{1}{u_i}\).
\end{itemize}
Moreover, this sequence satisfies
\[
\liminf_{n\to\infty} a_n^{\frac{1}{2^n}} \leq \lim_{n\to\infty} c_n^{\frac{1}{2^n}}.
\]
\end{thm}

As a direct consequence, we obtain the constructive proof of Conjecture \ref{conj1}.
\begin{cor}\label{maincor}
    Conjecture \ref{conj1} is ture.
\end{cor}

\subsection{Non-constructive Method}
We now continue our discussion by introducing additional terminology. It is known that for any given \( 0 < \lambda \leq 1 \), the \( n \)-term Egyptian underapproximation with the largest sum of reciprocals, referred to as \emph{the best \( n \)-term Egyptian underapproximation} of \( \lambda \), is guaranteed to exist by \cite[Theorem 3]{Nat23}.

Moreover, not every greedy \( n \)-term Egyptian underapproximation is necessarily the best \( n \)-term Egyptian underapproximation. For example, consider \( \lambda = \frac{10}{61} \). Its greedy two-term Egyptian underapproximation is \(\{7,48\}\), whereas the Egyptian underapproximation \(\{9,19\}\) satisfies  
\[
\frac{1}{7} + \frac{1}{48} < \frac{1}{9} + \frac{1}{19} < \frac{10}{61}.
\]
Thus, a natural question arises: what kind of rational numbers satisfy the property that the best \( n \)-term Egyptian underapproximation is always the greedy \( n \)-term Egyptian underapproximation? This problem has been partially explored by Soundararajan \cite{Sou05}, Nathanson \cite{Nat23}, and Chu \cite{Chu23}, where they investigated conditions on coprime integers \( p \) and \( q \) under which \( p/q \) satisfies this property. It is worth noting that in their results, the best \( n \)-term Egyptian underapproximation is always unique. For instance, Nathanson \cite[Theorem 5]{Nat23} proved that if a rational number \( p/q \) satisfies \( p \mid (q+1) \), then for every positive integer \( m \), the best \( m \)-term Egyptian underapproximation of \( p/q \) is unique and coincides with the greedy \( m \)-term Egyptian underapproximation.

A number \( 0 < \lambda \leq 1 \) is said to have \emph{eventually greedy best Egyptian underapproximations} if, for sufficiently large \( n \), the sum of reciprocals of its best \( n \)-term Egyptian underapproximation, denoted by \( R_n(\lambda) \)\footnote{For a given rational number \( \lambda \), its best \( n \)-term Egyptian underapproximation is not necessarily unique, but the sum of their reciprocals is always the same.}, satisfies  
\[
R_n\left(\lambda\right) = R_{n-1}\left(\lambda\right) + \frac{1}{m},
\]
where \( m \) is the smallest denominator not yet used such that \( R_n(\lambda) < \lambda \).

In their 1980 monograph \cite[p.~31]{EG80}, Erd\H{o}s and Graham stated:  
\begin{quote}
\emph{In other words, it is true that for any rational \( \frac{a}{b} \), the closest strict underapproximation \( R_n\left(\frac{a}{b}\right) \) of \( \frac{a}{b} \) by a sum of \( n \) unit fractions is given by  
\[
R_n\left(\frac{a}{b}\right) = R_{n-1}\left(\frac{a}{b}\right) + \frac{1}{m}
\]
where \( m \) is the least denominator not yet used for which \( R_n\left(\frac{a}{b}\right) < \frac{a}{b} \), provided
that \(n\) is sufficiently large.}
\end{quote}
Nevertheless, they provided neither a proof nor a reference. More than three decades later, Graham \cite[p.~296]{Gra13} implicitly withdrew this claim. He softened the assertion to:  
\begin{quote}
\emph{In other words, is it true that for any rational \( \frac{a}{b} \), the closest strict underapproximation \( R_n\left(\frac{a}{b}\right) \) of \( \frac{a}{b} \) is given by 
\[
R_n\left(\frac{a}{b}\right) = R_{n-1}\left(\frac{a}{b}\right) + \frac{1}{m}
\]
where \( m \) is the least denominator not yet used for which \( R_n\left(\frac{a}{b}\right) < \frac{a}{b} \) provided
that \(n\) is sufficiently large? }
\end{quote}
and subsequently reformulated it as an open question. Nathanson later revisited this problem, explicitly listing it as an open question in \cite[\S~Open problems, (4)]{Nat23}. Although Kova\v{c} \cite[Theorem 1]{Kovac2025} proved that the set of real numbers with this property has Lebesgue measure zero, this result alone is insufficient to settle the unproven claim. We restate this claim here in the form of a conjecture:
\begin{conj}[\cite{EG80,Gra13,Nat23}]\label{conj2}
For every rational number \( 0 < \theta \leq 1\), there exists an integer \( n_0 = n_0(\theta) \) such that, for all \( n \geq n_0 + 1 \),  
\[
R_n(\theta) = R_{n_0}(\theta) + R_{n-n_0}\left( \theta - R_{n_0}(\theta) \right),
\]
and the best \( (n-n_0) \)-term underapproximation \( R_{n-n_0}\left( \theta - R_{n_0}(\theta) \right) \) 
is always constructed by the greedy algorithm.
\end{conj}

In fact, assuming that Conjecture \ref{conj2} holds, we can establish a stronger result as a generalization of Conjecture \ref{conj1}:

\begin{thm}\label{mainthm2}
Assume that Conjecture \ref{conj2} holds. Let \( 0 < \lambda \leq 1 \) be a rational number whose best \( m \)-term Egyptian underapproximation is unique for every positive integer \( m \). Let \( \{u_n\}_{n=1}^{\infty} \) denote the best Egyptian underapproximation of \( \lambda \), and let \( a_1 < a_2 < \cdots \) be any other sequence of positive integers satisfying 
\[
\sum_{i=1}^\infty \frac{1}{a_i} = \lambda.
\]  
Then,
\[
\liminf_{n\to\infty} a_n^{\frac{1}{2^n}} < \lim_{n\to\infty} u_n^{\frac{1}{2^n}}.
\]
\end{thm}

By combining Nathanson \cite[Theorem 5]{Nat23} and Chu \cite[Theorem 1.12]{Chu23}, we obtain the following direct corollary:

\begin{cor}\label{corpq}
Assume that Conjecture \ref{conj2} holds. Let \( 0 < \frac{p}{q} \leq 1 \) be an irreducible fraction satisfying one of the following conditions:
\begin{enumerate}
    \item \( p \) divides \( q+1 \);
    \item \( q \) is odd, and \( l = 2 \) is the smallest positive integer such that \( p \) divides \( q + l \).
\end{enumerate}
Let \( \{u_n\}_{n=1}^{\infty} \) denote the infinite greedy Egyptian underapproximation of \( \frac{p}{q} \), and let \( a_1 < a_2 < \cdots \) be any other sequence of positive integers satisfying  
\[
\sum_{i=1}^\infty \frac{1}{a_i} = \frac{p}{q}.
\]  
Then,  
\[
\liminf_{n\to\infty} a_n^{\frac{1}{2^n}} < \lim_{n\to\infty} u_n^{\frac{1}{2^n}}.
\]
\end{cor}

For the non-constructive approach to Theorem \ref{mainthm2}, we first need to establish the following result:

\begin{thm}\label{mainthm1}
Assume that Conjecture \ref{conj2} holds. Let \( 0 < \lambda \leq 1 \) be a rational number whose best \( m \)-term Egyptian underapproximation is unique for every positive integer \( m \). Let \( \{u_n\}_{n=1}^{\infty} \) be the best Egyptian underapproximation of \( \lambda \), and let \( \{a_n\}_{n=1}^{\infty} \) be an eventually Sylvester sequence of positive integers satisfying  
\[
\sum_{i=1}^\infty \frac{1}{a_i} = \lambda.
\]  
If \( \{a_n\}_{n=1}^{\infty} \) and \( \{u_n\}_{n=1}^{\infty} \) are distinct as sets, then  
\[
\liminf_{n\to\infty} a_n^{\frac{1}{2^n}} = \lim_{n\to\infty} a_n^{\frac{1}{2^n}}< \lim_{n\to\infty} u_n^{\frac{1}{2^n}}.
\]
\end{thm}

To fully establish Theorem \ref{mainthm2}, it remains to prove the following existence result:
\begin{thm}\label{mainthm3}
Assume that Conjecture \ref{conj2} holds. Let \( 0 < \lambda \leq 1 \) be a rational number whose best \( m \)-term Egyptian underapproximation is unique for every positive integer \( m \). Let \( \{u_n\}_{n=1}^{\infty} \) be the best Egyptian underapproximation of \( \lambda \), and let \( a_1 < a_2 < \cdots \) be any other sequence of positive integers satisfying  
\[
\sum_{i=1}^\infty \frac{1}{a_i} = \lambda.
\]  
Then, there exists an eventually Sylvester sequence \(\{c_n\}_{n=1}^{\infty}\) of positive integers satisfying \(\sum_{i=1}^\infty \frac{1}{c_i} = \lambda\), such that \( \{c_n\}_{n=1}^{\infty} \) and \( \{u_n\}_{n=1}^{\infty} \) are distinct as sets, and  
\[
\liminf_{n\to\infty} a_n^{\frac{1}{2^n}} \leq \lim_{n\to\infty} c_n^{\frac{1}{2^n}}.
\]
\end{thm}

\begin{remark}
Shortly after this paper was submitted for publication, the authors became aware that a proof of Conjecture \ref{conj1} had also been obtained independently and almost simultaneously by Yuhi Kamio \cite{Kamio2025}. His proof is also constructive but differs from ours in many important ways. For instance, he introduced the concept of the so-called "Generalized Sylvester sequence" and employed a completely different approach to proving the "strict monotonicity" property, analogous to Lemma \ref{lemmono} in our work.   
\end{remark}

\subsection*{Organization}  
In Section~\ref{sec2}, we introduce the preliminaries needed for the rest of the paper. In Section~\ref{sec3}, we present the constructive proof of Conjecture~\ref{conj1}. In Section~\ref{sec4}, we provide the non-constructive proof of Theorem~\ref{mainthm2}. Finally, we propose several open problems for further study in Section~\ref{sec:conclusion}.

\subsection*{Acknowledgements}
The authors would like to express their sincere gratitude to Prof. Vjekoslav Kova\v{c} for his valuable comments, which have significantly improved an earlier version of this paper. They are also grateful to Thomas Bloom for founding and maintaining the website \cite{EP}.

\section{Preliminaries}\label{sec2}

\subsection{The Greedy Approximation and Underapproximation}

The term "approximation" corresponds to "underapproximation". For a positive integer \( n \), we say that a positive integer sequence \( \{x_i\}_{i=1}^n \) is an \emph{\( n \)-term Egyptian approximation} of a rational number \( 0 < \lambda \leq 1 \) if  
\begin{equation}\label{eq:nounder}
\sum_{i=1}^{n} \frac{1}{x_i} = \lambda,
\end{equation}
where \( 2 \leq x_1 \leq x_2 \leq \cdots \leq x_n \). We also define a positive integer sequence \( \{x_i\}_{i=1}^{\infty} \) as an \emph{infinite Egyptian approximation} of a rational number \( 0 < \lambda \leq 1 \) if \(\sum_{i=1}^{\infty} \frac{1}{x_i} = \lambda\), where \(2 \leq x_1 \leq x_2 \leq \cdots  \). Clearly, an infinite Egyptian approximation of \(\lambda\) is necessarily an infinite Egyptian underapproximation of \(\lambda\).

Similar to the Egyptian underapproximation, the greedy algorithm is also an effective method for obtaining an Egyptian approximation. We define the empty sum \( \sum_{i=1}^{0} \) to be \( 0 \). The \emph{greedy \( n \)-term Egyptian approximation} of a rational number \( 0 < \lambda \leq 1 \) is the sequence \( \{x_i\}_{i=1}^n \) satisfying \(\sum_{i=1}^{n} \frac{1}{x_i} = \lambda\), with each term \( x_m \) constructed recursively as follows:  
\[
\sum_{i=1}^{m} \frac{1}{x_i} = \sum_{i=1}^{m-1} \frac{1}{x_i} + \frac{1}{k},
\]  
where \( k \geq 2 \) is the smallest denominator not yet used such that \(\sum_{i=1}^{m} \frac{1}{x_i} \leq \lambda.\)

\begin{exm}
For \( \lambda = 1 \), its greedy Egyptian approximation is given by \(\{2, 3, 6\}\):
\[1 = \frac{1}{2} + \frac{1}{3} + \frac{1}{6}.\]
\end{exm}

In fact, the greedy Egyptian approximation of any rational number is necessarily finite. More precisely, we have the following proposition, which was first provided by Fibonacci in 1202 (see \cite[p.~322]{Wagon1999}):
\begin{prop}\label{propfinite}
    The greedy Egyptian approximation algorithm for a rational number \( 0 < \frac{a}{b} \leq 1 \) always terminates in a finite number of steps.
\end{prop}

\begin{proof}
Let \( 0 < \frac{a}{b} \leq 1 \) be a given rational number. The denominator \( t \) of the largest unit fraction that does not exceed \( \frac{a}{b} \) is given by  
    \[
   t = \max\left\{\left\lceil \frac{b}{a} \right\rceil, 2\right\}.
    \]  
    The algorithm proceeds recursively by subtracting \( \frac{1}{t} \) from \( \frac{a}{b} \) and applying the same procedure to the remainder:  
    \[
    \frac{a}{b} - \frac{1}{t}.
    \]  
    Since  
    \[
    \frac{1}{t} \leq \frac{a}{b} < \frac{1}{t-1},
    \]  
    it follows that the numerator of the remaining fraction,  
    \[
    ta - b,
    \]  
    is strictly smaller than \( a \).  

    This process generates a strictly decreasing sequence of numerators, all of which are positive integers. Hence, the sequence must eventually reach zero, ensuring that the algorithm terminates in a finite number of steps.
\end{proof}

A natural question arises: If a sequence \( \{x_n\}_{n=1}^{\infty} \) is the infinite greedy Egyptian underapproximation of a rational number, what structural properties does it possess? The following proposition provides such a characterization:

\begin{prop}\label{thmstructure}
Let \( 0 < \lambda \leq 1 \) be a rational number. If a sequence of positive integers \( \{x_n\}_{n=1}^{\infty} \) is the infinite greedy Egyptian underapproximation of \( \lambda \), then there exists a positive integer \( N \) such that the terms of the sequence \( \{x_n\}_{n=1}^{\infty} \) can be classified into the following two categories:
    \begin{enumerate}
        \item \textbf{The Finite Greedy Approximation Component:} The terms
        \[
        x_{1}, x_{2}, \dots, x_{N-1}, x_{N}-1,
        \]
        form the greedy \( N \)-term Egyptian approximation for \( \lambda \). That is,  
        \[
        \lambda = \frac{1}{x_1} + \frac{1}{x_2} + \cdots +\frac{1}{x_{N-1}} + \frac{1}{x_{N}-1}
        \]
        is an Egyptian fraction obtained through the greedy algorithm.

        \item \textbf{The Infinite Greedy Underapproximation Component of a Unit Fraction:} 
        The terms
        \[
        x_{N}, x_{N+1}, x_{N+2}, x_{N+3}, \dots
        \]
        form the infinite greedy Egyptian underapproximation of the unit fraction
        \[
        \frac{1}{x_{N}-1}.
        \]
    \end{enumerate}
\end{prop}
\begin{proof} By Proposition \ref{propfinite}, since the greedy Egyptian approximation of the rational number \( \lambda \) has only finitely many terms, there exists a positive integer \( N \) such that \(\{y_1, y_2, \cdots, y_{N-1}, y_{N}\}\) is the greedy \( N \)-term Egyptian approximation for \( \lambda \):
    \[
    \lambda = \frac{1}{y_1} + \frac{1}{y_2} + \dots + \frac{1}{y_{N-1}} + \frac{1}{y_{N}}.
    \]    
    Since \( \{x_n\}_{n=1}^{\infty} \) is the infinite greedy Egyptian underapproximation to \( \lambda \), we have \( x_k = y_k \) for \( 1 \leq k \leq N - 1 \) and \( x_{N} > y_{N} \). Let \(\mu = \lambda - \frac{1}{y_1} - \frac{1}{y_2} - \dots - \frac{1}{y_{N-1}} = \frac{1}{y_N}\), then we know that 
    \[
    \mu = \frac{1}{y_N} = \frac{1}{x_N} + \frac{1}{x_{N+1}} + \frac{1}{x_{N+2}} + \frac{1}{x_{N+3}} + \cdots
    \] is the sum of reciprocals of the infinite greedy Egyptian underapproximation of \( \frac{1}{y_N} \), where  
\[
x_{N} < x_{N+1} < x_{N+2} < x_{N+3} < \cdots.
\] Thus, we must have \( x_{N} = y_{N} + 1 \), and the sequence  
\[
x_{N}, x_{N+1}, x_{N+2}, x_{N+3}, \dots
\]  
is precisely the infinite greedy Egyptian underapproximation of the unit fraction \( \frac{1}{x_N - 1} \). Meanwhile, the terms  
\[
x_{1}, x_{2}, \dots, x_{N-1}, x_{N}-1
\]  
form the greedy \( N \)-term Egyptian approximation for \( \lambda \). Therefore, the proof is complete.
\end{proof}

\begin{exm}\label{example1920}
For \( \lambda = \frac{19}{20} \), its greedy Egyptian approximation is given by 
\[
\frac{19}{20} = \frac{1}{2} + \frac{1}{3} + \frac{1}{9} + \frac{1}{180},
\]
and its infinite greedy Egyptian underapproximation is given by 
\[
\frac{19}{20} = \frac{1}{2} + \frac{1}{3} + \frac{1}{9} + \frac{1}{181} + \frac{1}{32581} + \frac{1}{1061488981} + \frac{1}{1126758855722929381} + \cdots.
\] Thus, the terms \( \{2,3,9,181-1\} \) form the \emph{Finite Greedy Approximation Component}, while the terms \(\{181, 32581, 1061488981, 1126758855722929381, \dots\}\) form the \emph{Infinite Greedy Underapproximation Component} of the unit fraction \( \frac{1}{180} \).
\end{exm}

As a corollary of Proposition \ref{thmstructure}, we can show that every infinite greedy Egyptian underapproximation of a rational number is necessarily eventually Sylvester.

\begin{cor}\label{lemfinalsylvester}
    Let \( 0 < \lambda \leq 1 \) be a rational number. If a sequence of positive integers \(\{x_n\}_{n=1}^{\infty}\) is the infinite greedy Egyptian underapproximation of \( \lambda \), then there exists a positive integer \( N \) such that, for all \( n \geqslant N \),
    \[
    x_{n+1} = x_n^2 - x_n + 1.
    \]
\end{cor}

\begin{proof}
    By Proposition \ref{thmstructure}, we know that there exists a positive integer \( N \) such that the terms
    \[
    x_{N}, x_{N+1}, x_{N+2}, x_{N+3}, \dots
    \]
    form the infinite greedy Egyptian underapproximation of the unit fraction \(\frac{1}{x_{N}-1}.\)

    It is known that the infinite greedy Egyptian underapproximation \( \{w_n\}_{n=1}^{\infty} \) of the unit fraction \( \frac{1}{t} \) satisfies the recurrence relation:
    \[
    w_1 = 1 + t, \quad w_{n+1} = 1 + t w_1 w_2 \dots w_n.
    \]
    By induction, one can easily verify the following recurrence formula (see \cite[Equation (1.2)]{Badea1993}):
    \[
    w_{n+1} = w_n^2 - w_n + 1.
    \] Let \( \{z_n\}_{n=1}^{\infty} \) be the infinite greedy Egyptian underapproximation of the unit fraction \( \frac{1}{x_{N}-1} \), then it follows that
    \[
    z_{n+1} = z_n^2 - z_n + 1, \quad \text{where } z_k = x_{k+N-1}.
    \]
    Therefore, for all \( n \geq N \), we have
    \[
    x_{n+1} = x_n^2 - x_n + 1.
    \] \end{proof}

\begin{exm}
In Example \ref{example1920}, we have \( N = 4 \), and the terms  
\[
\{181, 32581, 1061488981, 1126758855722929381, \dots\}
\] form the infinite greedy underapproximation of the unit fraction \( \frac{1}{180} \). It is easy to verify that  
\begin{align*}
32581 &= 181^2 - 181 + 1,\\
1061488981 &= 32581^2 - 32581 + 1,\\
1126758855722929381 &= 1061488981^2 - 1061488981 + 1.
\end{align*}

\end{exm}

\subsection{Auxiliary Lemmas}
First, we recall the following result from Nathanson \cite{Nat23}:
\begin{prop}[\cite{Nat23}, Theorem 4]\label{lemNathanson} If \( \{x_i\}_{i=m+1}^{n} \) and \( \{a_i\}_{i=m+1}^{n} \) are increasing sequences of positive numbers such that
\[
\prod_{i=m+1}^{m+k} a_i \leq \prod_{i=m+1}^{m+k} x_i
\]
for all \( k = 1, \dots, n - m \), then
\[
\sum_{i=m+1}^{n} \frac{1}{x_i} \leq \sum_{i=m+1}^{n} \frac{1}{a_i}.
\]
\end{prop}

The following proposition is an analogue of Nathanson \cite[Theorem 5]{Nat23}. Its proof follows directly from the definition and is therefore immediate.

\begin{prop}\label{lem:Nathanson}
Let \( 0 < \lambda \leq 1 \) be a rational number whose best \( m \)-term Egyptian underapproximation is unique for every positive integer \( m \). Let \( \{u_n\}_{n=1}^{\infty} \) denote the best Egyptian underapproximation of \( \lambda \). For every positive integer \( n \), let \( \{a_i\}_{i=1}^{n} \) be a sequence of positive integers satisfying \(\sum_{i=1}^{n} \frac{1}{a_i} < \lambda\), with \( a_1 \leq a_2 \leq \cdots \leq a_n \). Then  
\begin{equation}\label{ineq1}
    \sum_{i=1}^{n} \frac{1}{a_i} \leq \sum_{i=1}^{n} \frac{1}{u_i},
\end{equation}  
and equality in (\ref{ineq1}) holds if and only if \( a_i = u_i \) for all \( i = 1, \dots, n \).
\end{prop}

Now, we introduce the concept of the \textit{increasing rearrangement} of a sequence.
\begin{defn}
Let \(\{a_i\}_{i=1}^{n}\) be a sequence of positive integers, where all elements are distinct and the sequence is not necessarily monotonically increasing. We define its \textit{increasing rearrangement} \(\{b_i\}_{i=1}^{n}\) as the sequence satisfying  
\[
b_1 \leq b_2 \leq \cdots \leq b_n,
\]
where \(\{b_i\}_{i=1}^{n}\) and \(\{a_i\}_{i=1}^{n}\) are identical as sets.
\end{defn}

In Proposition \ref{lem:Nathanson}, we note that the sequence \(\{a_i\}_{i=1}^{n}\) is inherently monotonically increasing by definition. As a direct corollary, we present an unordered version of Proposition \ref{lem:Nathanson}, which will be useful in Section \ref{sec41}:
\begin{cor}\label{cor:Nathanson}
Let \( 0 < \lambda \leq 1 \) be a rational number whose best \( m \)-term Egyptian underapproximation is unique for every positive integer \( m \). Let \( \{u_n\}_{n=1}^{\infty} \) denote the best Egyptian underapproximation of \( \lambda \). For every positive integer \( k \), let \( \{a_i\}_{i=1}^{k} \) be a sequence of positive integers satisfying \(\sum_{i=1}^{k} \frac{1}{a_i} < \lambda\).
Let \( \{b_i\}_{i=1}^{k} \) denote the increasing rearrangement of \( \{a_i\}_{i=1}^{k} \). Then  
\begin{equation}\label{ineq2}
    \sum_{i=1}^{k} \frac{1}{a_i} \leq \sum_{i=1}^{k} \frac{1}{u_i},
\end{equation}  
and equality in (\ref{ineq2}) holds if and only if \( b_i = u_i \) for all \( i = 1, \dots, k \).
\end{cor}

The following lemma is crucial for proving the "strictly less than" part of Theorems \ref{mainthm} and \ref{mainthm1}. It establishes the strict monotonicity of the limits of different eventually Sylvester sequences with respect to their first terms.

\begin{lem}\label{lemmono}
Let \( a_1 \geq 3 \) be a real number, and let the sequence \( \{a_n\}_{n=1}^{\infty} \) satisfy the recurrence relation:
\[
2a_{n+1} = a_n^2 + 1.
\]
Define the function:
\[
f(a_1) = \log a_1 + \sum_{j=1}^{\infty} 2^{-j} \log \left( 1 + \frac{1}{a_j^2} \right).
\] Then \(f(a_1)\) is bounded above. 

Moreover, if \( \{b_n\}_{n=1}^{\infty} \) is another sequence satisfying \( b_1 > a_1 \) and \( 2b_{n+1} = b_n^2 + 1\), then \[f(b_1)>f(a_1).\]
\end{lem}
\begin{proof}
Let \(
g(x) = \frac{x^2+1}{2}\). Denote \(
g^{(n)}(x) = \underbrace{g \circ g \circ \dots \circ g}_{n \text{ times}}(x)\), and define \( g^{(0)}(x) = x \).

Define the function:
\[
f(x) = \log(x) + \sum_{j=1}^{\infty} 2^{-j} \log \left( 1 + \frac{1}{g^{(j)}(x)^2} \right).
\]

First of all, we prove that \( f(x) \) is well-defined in \( [3, +\infty) \). For \( x \geq 3 \), since \( g^{(j)}(x) > 0 \), we know that each summand in \( f(x) \) is positive. Moreover, since \(
g(x) = \frac{x^2+1}{2} \geq x\), it follows that
\[
g^{(j)}(x) \geq g^{(j-1)}(x) \geq \cdots \geq x \geq 3.
\]
Hence, we know that
\[
\log \left( 1 + \frac{1}{g^{(j)}(x)^2} \right) \leq \log \left( \frac{10}{9} \right) < 1.
\] Therefore,
\[
f(x) \leq \log (x) + \sum_{j=1}^{\infty} 2^{-j} = \log (x) + 1.
\]
This means that \( f(x) \) is convergent and well-defined for \( x \in [3, +\infty) \), and is also bounded above.

Next, we will prove that \( f(x) \) is a strictly increasing function on \( [3, +\infty) \).

We know the fact that:
\[
\frac{d}{dx} g^{(n)}(x) = \prod_{k=0}^{n-1} g'(g^{(k)}(x)),
\]
then by the dominated convergence theorem, we can differentiate \( f(x) \): \begin{align*}
f'(x) &= \frac{1}{x} + \sum_{j=1}^{\infty} 2^{-j} \cdot \frac{1}{1 + \frac{1}{g^{(j)}(x)^2}} \cdot \frac{-2}{g^{(j)}(x)^3} \cdot \prod_{k=0}^{j-1} g'(g^{(k)}(x))\\
&= \frac{1}{x} - \sum_{j=1}^{\infty} 2^{-(j-1)} \cdot \frac{\prod_{k=0}^{j-1} g'(g^{(k)}(x))}{g^{(j)}(x)^3 + g^{(j)}(x)}.
\end{align*} Since \( g'(x) = x \) and \(
\frac{1}{x} = \sum_{j=1}^{\infty} 2^{-j} \cdot \frac{1}{x}\), we know that
\[
f'(x) = \sum_{j=1}^{\infty} \frac{1}{2^{j-1}} \cdot \left( \frac{1}{2x} - \frac{\prod_{k=0}^{j-1} g^{(k)}(x)}{g^{(j)}(x)^3 + g^{(j)}(x)} \right).
\]
Let \(
h_j(x) = \frac{1}{2x} - \frac{\prod_{k=0}^{j-1} g^{(k)}(x)}{g^{(j)}(x)^3 + g^{(j)}(x)}\). Then we obtain that \(
f'(x) = \sum_{j=1}^{\infty} \frac{1}{2^{j-1}} \cdot h_j(x)\).

Now, we claim that for each \( j \geq 1 \), we have \( h_j(x) > 0 \) when \( x \geq 3 \).
We shall prove this claim by induction.

When \( j = 1 \), we have
\[
h_1(x) = \frac{1}{2x} - \frac{x}{\left( \frac{x^2+1}{2} \right)^3 + \frac{x^2+1}{2}} > 0 \quad \text{for } x \geq 3,
\]
which holds obviously.

Assume \( h_j(x) > 0 \) for \( 1 \leq j \leq n \). Consider \begin{align*}
h_{n+1}(x) &= \frac{1}{2x} - \frac{\prod_{k=0}^{n} g^{(k)}(x)}{g^{(n+1)}(x)^3 + g^{(n+1)}(x)}\\
&= \frac{1}{2x} - \frac{\prod_{k=0}^{n-1} g^{(k)}(x)}{g^{(n)}(x)^3 + g^{(n)}(x)} \cdot \frac{g^{(n)}(x) \cdot (g^{(n)}(x)^3 + g^{(n)}(x))}{g^{(n+1)}(x)^3 + g^{(n+1)}(x)}.
\end{align*} Let \( y = g^{(n)}(x) \), then \( g^{(n+1)}(x) = g(y) = \frac{y^2+1}{2} \). Additionally, we know from \( x \geq 3 \) that \( y \geq 3 \). Hence,
\[
\frac{g^{(n)}(x) \cdot (g^{(n)}(x)^3 + g^{(n)}(x))}{g^{(n+1)}(x)^3 + g^{(n+1)}(x)} = \frac{y(y^3+y)}{\left(\frac{y^2+1}{2}\right)^3 + \frac{y^2+1}{2}} < 1 \quad \text{for } y \geq 3.
\] It follows that
\[
h_{n+1}(x) > \frac{1}{2x} - \frac{\prod_{k=0}^{n-1} g^{(k)}(x)}{g^{(n)}(x)^3 + g^{(n)}(x)} = h_n(x) > 0,
\] which completes the proof of the claim that \( h_j(x) > 0 \) for \( j \geq 1 \).
Therefore,
\[
f'(x) = \sum_{j=1}^{\infty} \frac{1}{2^j} \cdot h_j(x) > 0 \quad \text{for } x \geq 3.
\] This means that \( f(x) \) is a strictly increasing function, thus we have completed the proof that
\[
f(b_1) > f(a_1),
\]
where \( b_1 > a_1 \geq 3 \). \end{proof}

\begin{lem}\label{lemlimexist}
Let \( \{c_n\}_{n=1}^{\infty} \) be an eventually Sylvester sequence satisfying \( c_k \geq 2 \) for all positive integers \( k \). Then, the limit  
    \[
    \lim_{n\to\infty} c_n^{\frac{1}{2^n}}
    \]  
    exists.
\end{lem}
\begin{proof} By the definition of eventually Sylvester sequence, there exists an integer \( N > 0 \) such that for \( n \geq N \), the sequence \( \{c_n\}_{n=1}^\infty \) satisfies the following recurrence relations:
\[c_{n+1} = c_n^2 - c_n + 1.\]
Now we only need to consider the behavior of \( c_n \) for \( n \geq N\). Since
\[
c_{n+1} - \frac{1}{2} = (c_n - \frac{1}{2})^2 + \frac{1}{4},
\]
it follows that
\[
\frac{c_{n+1} - \frac{1}{2}}{(c_n - \frac{1}{2})^2} = 1 + \frac{1}{4 (c_n - \frac{1}{2})^2}.
\]
Taking the natural logarithm on both sides, we obtain:
\[
\log (c_{n+1} - \frac{1}{2}) - 2 \log (c_n - \frac{1}{2}) = \log \left( 1 + \frac{1}{4 (c_n - \frac{1}{2})^2} \right).
\]
That is,
\[
2^{-(n+1)} \log (c_{n+1} - \frac{1}{2}) - 2^{-n} \log (c_n - \frac{1}{2}) = 2^{-(n+1)} \log \left( 1 + \frac{1}{4 (c_n - \frac{1}{2})^2} \right).
\]
Summing both sides from \( N \) to \(\infty \), we derive:
\[
\lim_{n \to \infty} 2^{-n} \log (c_n - \frac{1}{2}) - 2^{-N} \log (c_{N} - \frac{1}{2}) = \sum_{n=N}^{\infty} 2^{-(n+1)} \log \left( 1 + \frac{1}{4 (c_n - \frac{1}{2})^2} \right).
\]
Therefore, we know that:
\[
\lim_{n \to \infty} \log \left[ (c_n - \frac{1}{2})^{\frac{1}{2^n}} \right] = 2^{-N} \log (c_{N} - \frac{1}{2}) + \sum_{n=N}^{\infty} 2^{-(n+1)} \log \left( 1 + \frac{1}{4 (c_n - \frac{1}{2})^2} \right).
\]
That is,
\[
\lim_{n \to \infty} (c_n - \frac{1}{2})^{\frac{1}{2^n}} = \exp \left\{ 2^{-N} \log (c_{N} - \frac{1}{2}) + \sum_{n=N}^{\infty} 2^{-(n+1)} \log \left( 1 + \frac{1}{4 (c_n - \frac{1}{2})^2} \right) \right\}.
\]
Furthermore, since
\(
\lim_{n \to \infty} c_n^{\frac{1}{2^n}}  = \lim_{n \to \infty} (c_n - \frac{1}{2})^{\frac{1}{2^n}},\)
it follows that
\[
\lim_{n \to \infty} c_n^{\frac{1}{2^n}} = \exp \left\{ 2^{-N} \log (c_{N} - \frac{1}{2}) + \sum_{n=N}^{\infty} 2^{-(n+1)} \log \left( 1 + \frac{1}{4 (c_n - \frac{1}{2})^2} \right) \right\}.
\] Now, we define
\[
J = 2^{-N} \log (c_{N} - \frac{1}{2}) + \sum_{n=N}^{\infty} 2^{-(n+1)} \log \left( 1 + \frac{1}{4 (c_n - \frac{1}{2})^2} \right),
\]
then we know that \(
\lim_{n \to \infty} c_n^{\frac{1}{2^n}} = e^{J}\).
Since
\begin{align*}
2^{N} J +\log 2&= \log (2 c_{N} - 1) + \sum_{n=N}^{\infty} 2^{N-(n+1)} \log \left( 1 + \frac{1}{4 (c_n - \frac{1}{2})^2} \right)
\\& = \log (2 c_{N} - 1) + \sum_{j=1}^{\infty} 2^{-j} \log \left( 1 + \frac{1}{(2 c_{j+N-1} - 1)^2} \right),
\end{align*}
we define the sequence \(z_n = 2 c_{n+N-1} - 1\). Then, the sequence \( \{z_n\}_{n=1}^{\infty} \) satisfies the following recurrence relations:
\[
2 z_{n+1} = z_n^2 + 1.
\] Therefore, we know that
\begin{align*}
2^{N} J + \log 2 &= \log (2 c_{N} - 1) + \sum_{j=1}^{\infty} 2^{-j} \log \left( 1 + \frac{1}{(2 c_{j+N-1} - 1)^2} \right)\\&= \log (z_1) + \sum_{j=1}^{\infty} 2^{-j} \log \left( 1 + \frac{1}{z_j^2} \right).
\end{align*}
Since \( z_1 = 2c_N - 1 \geq 3 \), it follows from Lemma \ref{lemmono} that the following positive series  
\[
f(z_1) = \log (z_1) + \sum_{j=1}^{\infty} 2^{-j} \log \left( 1 + \frac{1}{z_j^2} \right)
\]
is bounded above. Hence, \( f(z_1) \) converges, which implies the existence of \( J \), and consequently,  
\[
\lim_{n \to \infty} c_n^{\frac{1}{2^n}}
\]
exists. \end{proof}

\section{Proof of Constructive Method}\label{sec3}

\subsection{Proof of Theorem \ref{mainthm}}
We are now ready to present the following
\begin{proof}[Proof of Theorem \ref{mainthm}]

Recall that for \( n \geq N \), the sequence \( \{a_n\}_{n=1}^{\infty} \) satisfies the following recurrence relation:
\[a_{n+1} = a_n^2 - a_n + 1.\]

\noindent
\textbf{Claim.} \( a_N < u_N \).

We prove this claim by contradiction. Suppose that \( a_N \geq u_N \). Then, for any \( n \geq N \), we have \[ \frac{1}{a_n} \leq \frac{1}{u_n},\] since both sequences \(\{a_n\}_{n=1}^{\infty}\) and \(\{u_n\}_{n=1}^{\infty}\) satisfy the same recurrence relation.

Therefore, for all positive integers \( k \geq N \), it follows that  
\begin{equation}\label{eqq1}
\sum_{j=N}^{k} \frac{1}{a_j} \leq \sum_{j=N}^{k} \frac{1}{u_j}.
\end{equation} 
In particular, we have  
\begin{equation}\label{eqq2}
\sum_{j=N}^{\infty} \frac{1}{a_j} \leq \sum_{j=N}^{\infty} \frac{1}{u_j}.
\end{equation} Combining inequality (\ref{eqq2}) with \(
\sum_{j=1}^{\infty} \frac{1}{a_j} = \sum_{j=1}^{\infty} \frac{1}{u_j} = 1,
\) it follows that  
\[
\sum_{j=1}^{N-1} \frac{1}{a_j} \geq \sum_{j=1}^{N-1} \frac{1}{u_j}.
\] However, this contradicts our assumption (**) that \(\sum_{i=1}^{N-1} \frac{1}{a_i} < \sum_{i=1}^{N-1} \frac{1}{u_i}\).

Therefore, we must have \( a_N < u_N \), completing the proof of our claim.

By Lemma \ref{lemlimexist}, the limits \( \lim_{n \to \infty} a_n^{\frac{1}{2^n}} \) and \( \lim_{n \to \infty} u_n^{\frac{1}{2^n}} \) both exist. Define  
\[
\lim_{n \to \infty} a_n^{\frac{1}{2^n}} = e^{J_1} \quad \text{and} \quad \lim_{n \to \infty} u_n^{\frac{1}{2^n}} = e^{J_2}.
\]  
According to the proof of Lemma \ref{lemlimexist}, we define  
\[
J_3 = 2^{N} J_1 + \log 2 \quad \text{and} \quad J_4 = 2^{N} J_2 + \log 2.
\]  
Moreover, they satisfy the following equalities:
\[
J_3 = \log (2 a_{N} - 1) + \sum_{j=1}^{\infty} 2^{-j} \log \left( 1 + \frac{1}{(2 a_{j+N-1} - 1)^2} \right),
\]
\[
J_4 = \log (2 u_{N} - 1) + \sum_{j=1}^{\infty} 2^{-j} \log \left( 1 + \frac{1}{(2 u_{j+N-1} - 1)^2} \right).
\] We also define the following two sequences:
\[
x_n = 2 a_{n+N-1} - 1, \quad y_n = 2 u_{n+N-1} - 1.
\]
Then, the sequences \( \{x_n\}_{n=1}^{\infty} \) and \( \{y_n\}_{n=1}^{\infty} \) satisfy the following recurrence relations:
\[
2 x_{n+1} = x_n^2 + 1, \quad 2 y_{n+1} = y_n^2 + 1.
\] Since \( a_{N} < u_{N} \), we conclude that \( x_1 < y_1 \).
Therefore, we know that
\begin{align*}
J_3 &= \log (2 a_{N} - 1) + \sum_{j=1}^{\infty} 2^{-j} \log \left( 1 + \frac{1}{(2 a_{j+N-1} - 1)^2} \right)\\&= \log (x_1) + \sum_{j=1}^{\infty} 2^{-j} \log \left( 1 + \frac{1}{x_j^2} \right).
\end{align*}
And similarly,
\[
J_4 = \log (y_1) + \sum_{j=1}^{\infty} 2^{-j} \log \left( 1 + \frac{1}{y_j^2} \right).
\]
Apply Lemma \ref{lemmono}, we obtain:
\[
J_3 < J_4,
\] since \(x_1 < y_1\). It follows that \( J_1 < J_2 \), and thus we have completed the proof that
\[
\liminf_{n\to\infty} a_n^{\frac{1}{2^n}} = \lim_{n \to \infty} a_n^{\frac{1}{2^n}} < \lim_{n \to \infty} u_n^{\frac{1}{2^n}}.
\]

\end{proof}

\subsection{Construction of \( c_n \)}\label{sectionconstruction} Let \( \{u_n\}_{n=1}^{\infty} \) be the Sylvester’s sequence, and let \( a_1 < a_2 < \cdots \) be any other sequence of positive integers satisfying \(
\sum_{k=1}^\infty \frac{1}{a_k} = 1.
\) Let \( m \) be the smallest positive integer \( s \) such that \( a_i = u_i \) for all \( 1 \leq i \leq s-1 \) and \( a_s > u_s \).  
The existence of \( m \) is guaranteed since \( \{a_n\}_{n=1}^{\infty} \) and \( \{u_n\}_{n=1}^{\infty} \) are distinct sequences. At this point, we define \( c_i = u_i \) for all \( 1 \leq i \leq s-1 \).

In conjunction with the identity  
\[
\sum_{i=1}^{m-1} \frac{1}{u_i} + \frac{1}{u_m - 1} = 1, \quad \text{for } m \geq 1,
\]
we obtain \( a_m \geq u_m + 1 \) and the following equalities:  
\[
\frac{1}{u_m - 1} = \frac{1}{a_m} + \frac{1}{a_{m+1}} + \frac{1}{a_{m+2}} + \frac{1}{a_{m+3}} + \cdots,
\]
\[
\frac{1}{u_m - 1} = \frac{1}{c_m} + \frac{1}{c_{m+1}} + \frac{1}{c_{m+2}} + \frac{1}{c_{m+3}} + \cdots.
\]
Now, we define \( c_m \), \( c_{m+1} \), and \( c_{m+2} \) as follows:
\[
\frac{1}{c_m} =
\begin{cases}
    \frac{1}{u_m}, & m = 1, \\
    \frac{1}{u_m + 1}, & m > 1.
\end{cases}
\]

\[
\frac{1}{c_{m+1}} =
\begin{cases}
    \frac{1}{u_m^2}, & m = 1, \\
    \frac{2}{u_m^2}, & m > 1.
\end{cases}
\]

\[
\frac{1}{c_{m+2}} =
\begin{cases}
    \frac{2}{u_m^3 (u_m - 1) + 1}, & m = 1, \\
    \frac{2}{u_m^2 (u_m^2 - 1) + 2}, & m > 1.
\end{cases}
\] Before defining \( \{c_n\}_{n=m+3}^{\infty} \), we first claim that
\[
\frac{1}{u_m - 1} - \frac{1}{c_m} - \frac{1}{c_{m+1}} - \frac{1}{c_{m+2}}
\]
is a unit fraction.

\noindent
\begin{claim}\label{claim1} \( \frac{1}{u_m - 1} - \frac{1}{c_m} - \frac{1}{c_{m+1}} - \frac{1}{c_{m+2}} = \frac{1}{(u_m - 1) c_m c_{m+1} c_{m+2}}\) is a unit fraction.
\end{claim}
\begin{proof}[Proof of Claim \ref{claim1}]
When \( m = 1 \), we compute:
\[
\frac{1}{u_m - 1} - \frac{1}{c_m} - \frac{1}{c_{m+1}} - \frac{1}{c_{m+2}} = 1 - \frac{1}{2} - \frac{1}{4} - \frac{2}{9} = \frac{1}{36} = 1 \times \frac{1}{2} \times \frac{1}{4} \times \frac{2}{9}.
\]
Thus, the claim holds for \( m = 1 \).

For \( m > 1 \), we compute:
\begin{align*}
\frac{1}{u_m - 1} - \frac{1}{c_m} - \frac{1}{c_{m+1}} - \frac{1}{c_{m+2}}
&= \frac{1}{u_m - 1} - \frac{1}{u_m+1} - \frac{2}{u_m^2} - \frac{2}{u_m^2 (u_m^2 - 1) + 2}
\\&= \frac{2}{u_m^2 - 1} - \frac{2}{u_m^2} - \frac{2}{u_m^2 (u_m^2 - 1) + 2}
\\&= \frac{2}{u_m^2 (u_m^2 - 1)} - \frac{2}{u_m^2 (u_m^2 - 1) + 2}
\\&= \frac{4}{u_m^2 (u_m^2 - 1) \cdot [u_m^2 (u_m^2 - 1) + 2]}
\\&=\frac{1}{u_m - 1} \cdot \frac{1}{c_m} \cdot \frac{1}{c_{m+1}} \cdot \frac{1}{c_{m+2}}.
\end{align*} Since \( 2 \) divides both \( u_m^2 (u_m^2 - 1) \) and \( u_m^2 (u_m^2 - 1) + 2 \), the resulting fraction is therefore a unit fraction.

Hence, Claim \ref{claim1} holds for all \( m \geq 1 \).\end{proof}

By Claim \ref{claim1}, we can define \( \{c_n\}_{n=m+3}^{\infty} \) as the infinite greedy Egyptian underapproximation of the unit fraction  
\[
\frac{1}{(u_m - 1) c_m c_{m+1} c_{m+2}}.
\]
Thus, the sequence \( \{c_n\}_{n=1}^{\infty} \) must satisfy the following recurrence relation (see \cite[Equation (1.2)]{Badea1993}):
\begin{equation}\label{eqconstruct}
c_{n+1} = c_n^2 - c_n + 1, \quad \text{for } n \geq m+3.    
\end{equation}

Clearly, the sequence \( \{c_n\}_{n=1}^{\infty} \) must satisfy the condition (*) in Theorem \ref{mainthm0}. Next, we prove that \( \{c_n\}_{n=1}^{\infty} \) also satisfies condition (**).

\noindent
\begin{claim}\label{claim2}\(\sum_{i=1}^{m+2} \frac{1}{c_i} < \sum_{i=1}^{m+2} \frac{1}{u_i}\).
\end{claim} 
\begin{proof}[Proof of Claim \ref{claim2}]
Since \( c_i = u_i \) for all \( 1 \leq i \leq m-1 \), it suffices to prove that  
\begin{equation}\label{con1}
\frac{1}{c_m} + \frac{1}{c_{m+1}} + \frac{1}{c_{m+2}}< \frac{1}{u_m} + \frac{1}{u_{m+1}} + \frac{1}{u_{m+2}}.
\end{equation}
For \( m = 1 \), it is clear that  
\[
\frac{1}{2} + \frac{1}{4} + \frac{2}{9}< \frac{1}{2} + \frac{1}{3} + \frac{1}{7}.
\]
Thus, we only need to consider the case when \( m > 1 \).  

Clearly, we have  
\[
\frac{1}{c_m} + \frac{1}{c_{m+1}} + \frac{1}{c_{m+2}} =  \frac{1}{u_m+1} + \frac{2}{u_{m}^2} + \frac{2}{u_m^2 (u_m^2 - 1) + 2}\]and\[ \frac{1}{u_m} + \frac{1}{u_{m+1}} + \frac{1}{u_{m+2}} = \frac{1}{u_m} + \frac{1}{u_{m}^2 - u_m +1} + \frac{1}{(u_{m}^2 - u_m+1)(u_{m}^2 - u_m) +1}.
\] Since the function  
\begin{align*}
F(x) &= \frac{1}{x} + \frac{1}{x^2 - x + 1} + \frac{1}{(x^2 - x+1)(x^2 - x) + 1} - \frac{1}{x + 1} - \frac{2}{x^2}  - \frac{2}{x^2(x^2-1)+2}\\&=\frac{( x - 2) (3x^4- 4x^3+ 3x^2-2x+2 )}
{x^2 (x + 1) (x^2- x+1) (x^4- x^2+2) (x^4- 2x^3+ 2x^2- x+1)}
\end{align*} is always positive for \( x \geq 3 \), we conclude the proof of Claim \ref{claim2}. \end{proof}

Before presenting the proof of Theorem \ref{mainthm0}, we establish a useful lemma, whose proof follows a similar approach to that of Nathanson \cite[Theorem 5]{Nat23}.

\begin{lem}\label{lemmacomstructc}
Let \( a_1 < a_2 < \cdots \) be any sequence of positive integers distinct from the Sylvester's sequence, satisfying \(\sum_{i=1}^\infty \frac{1}{a_i} = 1\). Let \( m \) be the smallest positive integer \( s \) such that \( a_i = u_i \) for all \( 1 \leq i \leq s-1 \) and \( a_s > u_s \). Define the sequence \( \{c_n\}_{n=1}^{\infty} \) as the construction given at the beginning of Section \ref{sectionconstruction}. Then, for any \( k \geq m \), we have 
\begin{equation}\label{eqlemma31}
\sum_{i=1}^{k} \frac{1}{a_i} \leq \sum_{i=1}^{k} \frac{1}{c_i}.    
\end{equation}
\end{lem}
\begin{proof}
Since \( c_i = u_i \) for all \( 1 \leq i \leq m-1 \), it suffices to prove that  
\begin{equation}\label{construct1}
\sum_{i=m}^{k} \frac{1}{a_i} \leq \sum_{i=m}^{k} \frac{1}{c_i}.
\end{equation} We proceed by mathematical induction.  

For \( k = m \), inequality (\ref{construct1}) holds trivially.  

For \( k = m+1 \), it suffices to prove that  
\begin{equation}\label{construct2}
\frac{1}{a_m} + \frac{1}{a_{m+1}} \leq \frac{1}{c_m} + \frac{1}{c_{m+1}}.
\end{equation}
For \( m = 1 \), it is clear that  
\[
\frac{1}{a_m} + \frac{1}{a_{m+1}} \leq \frac{1}{3} + \frac{1}{4} < \frac{1}{2} + \frac{1}{4} = \frac{1}{c_m} + \frac{1}{c_{m+1}}.
\]
For \( m \geq 2 \), if \(\frac{1}{a_m a_{m+1}} > \frac{1}{c_m c_{m+1}}\),  
then we obtain  
\begin{align*}
\frac{1}{u_m - 1} - \frac{1}{c_m} - \frac{1}{c_{m+1}}&= \frac{2}{u_m^2(u_m^2 - 1)} \\&= \frac{1}{(u_m - 1) c_m c_{m+1}} \\
&< \frac{1}{(u_m - 1) a_m a_{m+1}} \\
&\leq \frac{1}{u_m - 1} - \frac{1}{a_m} - \frac{1}{a_{m+1}}.
\end{align*} which means inequality (\ref{construct2}) holds. Thus, it remains to consider the case where \(\frac{1}{a_m a_{m+1}} \leq \frac{1}{c_m c_{m+1}}\).
Now, we have the following inequalities:  
\[
\frac{1}{a_m} \geq \frac{1}{a_{m+1}}, \quad \frac{1}{c_m} \geq \frac{1}{c_{m+1}}, \quad \frac{1}{a_m} \leq \frac{1}{c_m}, \quad \frac{1}{a_m a_{m+1}} \leq \frac{1}{c_m c_{m+1}}.
\]
Thus, applying Proposition \ref{lemNathanson}, we obtain that inequality (\ref{construct2}) holds.

Suppose that the inequality (\ref{construct1}) holds for all \( m \leq k < n \), where \( n \geq m+2 \).  
Now, we consider the case \( k = n \), that is, we need to prove that  
\begin{equation}\label{construct3}
\sum_{i=m}^{n} \frac{1}{a_i} \leq \sum_{i=m}^{n} \frac{1}{c_i}
\end{equation} holds.

If \( \prod_{i=m}^{n} \frac{1}{a_i} > \prod_{i=m}^{n} \frac{1}{c_i} \), then by Claim \ref{claim1}, we have
\begin{align*}
\frac{1}{u_m - 1} - \sum_{i=m}^{n} \frac{1}{c_i} &= \frac{1}{(u_m - 1) c_m c_{m+1} c_{m+2}} - \sum_{i=m+3}^{n} \frac{1}{c_i} \\&= \frac{1}{(u_m - 1) \prod_{i=m}^{n} c_i} \\&< \frac{1}{(u_m - 1) \prod_{i=m}^{n} a_i} \\&\leq \frac{1}{u_m - 1} - \sum_{i=m}^{n} \frac{1}{a_i},
\end{align*}
which means \eqref{construct3} holds. Thus, we only need to consider the exceptional case where \( \prod_{i=m}^{n} \frac{1}{a_i} \leq \prod_{i=m}^{n} \frac{1}{c_i} \). Define \( t \geq m \) as the largest integer such that
\begin{equation}\label{construct4}
\prod_{i=t}^{n} \frac{1}{a_i} \leq \prod_{i=t}^{n} \frac{1}{c_i}.
\end{equation}
We shall prove that for all \( j \in \{1, \dots, n-t\} \),
\begin{equation}\label{construct5}
\prod_{i=t}^{t+j} \frac{1}{a_i} \leq \prod_{i=t}^{t+j} \frac{1}{c_i}.
\end{equation}
If not, then there exists some \( s \in \{1, \dots, n-t-1\} \) such that
\[
\prod_{i=t}^{t+s} \frac{1}{c_i} < \prod_{i=t}^{t+s} \frac{1}{a_i}.
\]
Using (\ref{construct4}), we obtain:
\[
\prod_{i=t+s+1}^{n} \frac{1}{a_i} \leq \frac{\prod_{i=t}^{t+s} (a_i/c_i)}{\prod_{i=t+s+1}^{n} c_i} < \prod_{i=t+s+1}^{n} \frac{1}{c_i}.
\]
This contradicts the maximality of \( t \). This proves (\ref{construct5}).

Applying Proposition \ref{lemNathanson} to the increasing sequences \(\{a_i\}_{i=t}^n \) and \(\{c_i\}_{i=t}^n \), we obtain
\begin{equation}\label{construct6}
\sum_{i=t}^{n} \frac{1}{a_i} \leq \sum_{i=t}^{n} \frac{1}{c_i}
\end{equation} holds. And the induction hypothesis implies
\begin{equation}\label{construct7}
\sum_{i=m}^{t-1} \frac{1}{a_i} \leq \sum_{i=m}^{t-1} \frac{1}{c_i}
\end{equation} holds. Thus, combining (\ref{construct6}) and (\ref{construct7}), we conclude that (\ref{construct3}) holds. Hence, the proof is complete. \end{proof}

We are now ready to present the following
\begin{proof}[Proof of Theorem \ref{mainthm0}]
Let \( m \) be the smallest positive integer \( s \) such that \( a_i = u_i \) for all \( 1 \leq i \leq s-1 \) and \( a_s > u_s \). We then define the sequence \( \{c_n\}_{n=1}^{\infty} \) as the construction given at the beginning of Section \ref{sectionconstruction}. Thus, in conjunction with Eq.~(\ref{eqconstruct}) and Claim \ref{claim2}, we conclude that the sequence \( \{c_n\}_{n=1}^{\infty} \) satisfies both conditions (*) and (**), and moreover,  
\[
c_{n+1} = c_n^2 - c_n + 1, \quad \text{for } n \geq m+3.
\]

Now, using proof by contradiction, suppose that  
\[
\liminf_{n\to\infty} a_n^{\frac{1}{2^n}} > \lim_{n\to\infty} c_n^{\frac{1}{2^n}}.
\]
Then the set \(\{n \in \mathbb{Z}_{> 0} \mid a_n \leq c_n \}\) is finite, implying that there exists a positive integer \( N \geq m+3\) such that for all \( n > N \), we have \( a_n > c_n \). Hence, we obtain  
\begin{equation}\label{con2}
\sum_{i=N+1}^{\infty} \frac{1}{a_i} < \sum_{i=N+1}^{\infty} \frac{1}{c_i}.
\end{equation} 

Furthermore, setting \( k = N \) in inequality (\ref{eqlemma31}) of Lemma \ref{lemmacomstructc} yields
\begin{equation}\label{con3}
\sum_{i=1}^{N} \frac{1}{a_i} \leq \sum_{i=1}^{N} \frac{1}{c_i}.
\end{equation} Combining (\ref{con2}) and (\ref{con3}), we conclude that  
\[
\sum_{i=1}^{\infty} \frac{1}{a_i} = \sum_{i=1}^{N} \frac{1}{a_i} + \sum_{i=N+1}^{\infty} \frac{1}{a_i}
< \sum_{i=1}^{N} \frac{1}{c_i} + \sum_{i=N+1}^{\infty} \frac{1}{c_i}
= \sum_{i=1}^{\infty} \frac{1}{c_i} = 1.
\]
However, this contradicts the assumption that \(\sum_{i=1}^{\infty} \frac{1}{a_i} = 1.\) Thus, we have completed the proof that \[
\liminf_{n\to\infty} a_n^{\frac{1}{2^n}} \leq \lim_{n\to\infty} c_n^{\frac{1}{2^n}}.
\]

\end{proof}

\subsection{Proof of Corollary \ref{maincor}}
By Theorem \ref{mainthm0}, we can construct an eventually Sylvester sequence \( \{c_n\}_{n=1}^{\infty} \) of positive numbers that satisfies the recurrence relation  
\[
c_{n+1} = c_n^2 - c_n + 1, \quad \text{for } n \geq N,
\]
and the following two conditions:
\begin{itemize}
    \item[(*)] \(\sum_{i=1}^\infty \frac{1}{c_i} = 1\);
    \item[(**)] \(\sum_{i=1}^{N-1} \frac{1}{c_i} < \sum_{i=1}^{N-1} \frac{1}{u_i}\).
\end{itemize}
Moreover, this sequence satisfies  
\[
\liminf_{n\to\infty} a_n^{\frac{1}{2^n}} \leq \lim_{n\to\infty} c_n^{\frac{1}{2^n}}.
\]
Furthermore, by Theorem \ref{mainthm}, we have
\[
\liminf_{n\to\infty} c_n^{\frac{1}{2^n}} = \lim_{n \to \infty} c_n^{\frac{1}{2^n}} < \lim_{n \to \infty} u_n^{\frac{1}{2^n}}.
\]
Thus, we conclude that  
\[
\liminf_{n\to\infty} a_n^{\frac{1}{2^n}} < \lim_{n\to\infty} u_n^{\frac{1}{2^n}}.
\]
This completes the proof.\hfill $\square$

\section{Proof of Non-constructive Method}\label{sec4}
\subsection{Proof of Theorem \ref{mainthm1}}\label{sec41}
We are now ready to present the following
\begin{proof}[Proof of Theorem \ref{mainthm1}]

By the definition of an eventually Sylvester sequence, there exists a positive integer \( N_1 > 0 \) such that \( \{a_n\}_{n=1}^{\infty} \) satisfies the recurrence relation:  
\[
a_{n+1} = a_n^2 - a_n + 1, \quad \text{for } n \geq N_1.
\] Since Conjecture \ref{conj2} holds, the best Egyptian underapproximation \( \{u_n\}_{n=1}^{\infty} \) of \( \lambda \) eventually becomes the infinite greedy Egyptian underapproximation of the rational number \( \lambda - R_{n_0(\lambda)}(\lambda) \). Thus, by combining this with Corollary \ref{lemfinalsylvester}, we conclude that \( \{u_n\}_{n=1}^{\infty} \) also satisfies the recurrence relation:
\[
u_{n+1} = u_n^2 - u_n + 1, \quad \text{for } n \geq N_2,
\]
where \( N_2 > 0 \) is a positive integer.

Let \( N = \max\{N_1, N_2\} \). Note that \( N \geq 2 \) because otherwise the sequences \( \{a_n\}_{n=1}^{\infty} \) and \( \{u_n\}_{n=1}^{\infty} \) would coincide entirely, contradicting the assumption that they are distinct.

\noindent
\textbf{Claim.} \( a_N < u_N \).

We prove this claim by contradiction. Suppose that \( a_N \geq u_N \). Then, for any \( n \geq N \), we have \[ \frac{1}{a_n} \leq \frac{1}{u_n},\] since both sequences \(\{a_n\}_{n=1}^{\infty}\) and \(\{u_n\}_{n=1}^{\infty}\) satisfy the same recurrence relation.

Therefore, for all positive integers \( k \geq N \), it follows that  
\begin{equation}\label{eeqq1}
\sum_{i=N}^{k} \frac{1}{a_i} \leq \sum_{i=N}^{k} \frac{1}{u_i}.
\end{equation} 
In particular, we have  
\begin{equation}\label{eeqq2}
\sum_{i=N}^{\infty} \frac{1}{a_i} \leq \sum_{i=N}^{\infty} \frac{1}{u_i}.
\end{equation} Combining inequality (\ref{eeqq2}) with \(
\sum_{i=1}^{\infty} \frac{1}{a_i} = \sum_{i=1}^{\infty} \frac{1}{u_i} = 1,
\) it follows that  
\[
\sum_{i=1}^{N-1} \frac{1}{a_i} \geq \sum_{i=1}^{N-1} \frac{1}{u_i}.
\]

Now, since the best Egyptian underapproximation \( \{u_n\}_{n=1}^{\infty} \) of \( \lambda \) is unique, it follows from Corollary \ref{cor:Nathanson} that  
\[
b_i = u_i \quad \text{for } 1 \leq i \leq N-1,
\]  
where \( \{b_i\}_{i=1}^{N-1} \) is the increasing rearrangement of \( \{a_i\}_{i=1}^{N-1} \). Thus,  
\begin{equation}\label{eeqq3}
\sum_{i=1}^{N-1} \frac{1}{a_i} = \sum_{i=1}^{N-1} \frac{1}{b_i} = \sum_{i=1}^{N-1} \frac{1}{u_i}.
\end{equation} Combining (\ref{eeqq2}) and (\ref{eeqq3}), we get  
\begin{equation}\label{eeqq4}
1 = \sum_{i=1}^{\infty} \frac{1}{a_i} \leq \sum_{i=1}^{\infty} \frac{1}{u_i} = 1.
\end{equation}
Since the inequality in (\ref{eeqq4}) must hold as an equality, it follows that (\ref{eeqq2}) also holds as an equality. That is, we have \begin{equation}\label{eeqq5}
\sum_{i=N}^{\infty} \frac{1}{a_j}=\sum_{i=N}^{\infty} \frac{1}{u_j}.\end{equation}

If \( a_N = u_N \), then \( a_i = u_i \) for all \( i \geq N \), since both sequences \(\{a_n\}_{n=N}^{\infty}\) and \(\{u_n\}_{n=N}^{\infty}\) satisfy the same recurrence relation. This contradicts the assumption that \(\{b_n\}_{n=1}^{\infty}\) and \(\{u_n\}_{n=1}^{\infty}\) are distinct sequences. If \( a_N > u_N \), then \( a_i > u_i \) for all \( i \geq N \), which contradicts Eq.~(\ref{eeqq5}).

Therefore, we must have \( a_N < u_N \), completing the proof of this claim.

The remaining steps are identical to those in the proof of Theorem \ref{mainthm}, thus completing the proof of Theorem \ref{mainthm1}. \end{proof}

\subsection{Proof of the Existence of \( c_n \)}\label{secexist}

We are now ready to present the following
\begin{proof}[Proof of Theorem \ref{mainthm3}]
Let \( N_1 \) be the smallest positive integer \(s\) such that \( a_s \notin \{u_k \mid k \geq 1\} \). The existence of \( N_1 \) is guaranteed since \(\{a_n\}_{n=1}^{\infty}\) and \(\{u_n\}_{n=1}^{\infty}\) are distinct sequences.

Then, we define  
\[
c_i = a_i \quad \text{for } 1 \leq i \leq N_1.
\] Observing that  
\[
\lambda = \frac{1}{a_1} + \frac{1}{a_2} + \cdots + \frac{1}{a_{N_1}} + \frac{1}{a_{N_1+1}} + \frac{1}{a_{N_1+2}} + \cdots,
\]
we define \(\mu = 1 - \sum_{i=1}^{N_1} \frac{1}{a_i}\), then we obtain  
\[
\mu = \frac{1}{a_{N_1+1}} + \frac{1}{a_{N_1+2}} + \frac{1}{a_{N_1+3}} + \cdots.
\]

Since the best \( n \)-term Egyptian underapproximation of any given \( x \in (0, 1) \) exists (see \cite[Theorem 3]{Nat23}), the best \( n \)-term Egyptian underapproximation of \(\mu\) must exist. Since Conjecture \ref{conj2} holds, we obtain the existence of an integer \( n_0 = n_0(\mu) \) such that, for all \( n \geq n_0 + 1 \),  
\begin{equation}\label{eqconj2}
R_n(\mu) = R_{n_0}(\mu) + R_{n-n_0}\left( \mu - R_{n_0}(\mu) \right),
\end{equation}
and the best \( (n-n_0) \)-term underapproximation \( R_{n-n_0}\left( \mu - R_{n_0}(\mu) \right) \)  
is always constructed by the greedy algorithm.

We now define the sequence \(\{c_n\}_{n=N_1+1}^\infty\) as follows:  
Define \(\frac{1}{c_{N_1+1}} + \frac{1}{c_{N_1+2}} + \cdots + \frac{1}{c_{N_1+n_0}}\) to be \( R_{n_0}(\mu) \). And starting from \( c_{N_1+n_0} \), each subsequent term is obtained from the previous one using the greedy algorithm.

Thus, by Eq.~(\ref{eqconj2}), the sequence \(\{c_i\}_{i=N_1+1}^{N_1+n}\) is the best \( n \)-term Egyptian underapproximation of \( \mu \) for \( n \geq n_0 + 1 \). Hence, the inequality  
\begin{equation}\label{ineq4}
\frac{1}{a_{N_1+1}} + \frac{1}{a_{N_1+2}} + \cdots + \frac{1}{a_{N_1+n}}    \leq    \frac{1}{c_{N_1+1}} + \frac{1}{c_{N_1+2}} + \cdots + \frac{1}{c_{N_1+n}}
\end{equation} holds for every \( n \geq n_0 + 1 \).

In this way, the sequence \( \{c_n\}_{n=1}^{\infty} \) we have defined not only satisfies \(\sum_{i=1}^\infty \frac{1}{c_i} = \lambda\), but is also distinct from \( \{u_n\}_{n=1}^{\infty} \) as a set, since \( c_{N_1} = a_{N_1} \notin \{u_k \mid k \geq 1\} \). Moreover, since \(\{c_n\}_{n=1}^{\infty}\) eventually becomes the infinite greedy Egyptian underapproximation of the rational number \(\mu - R_{n_0}(\mu)\), it follows that \(\{c_n\}_{n=1}^{\infty}\) is eventually Sylvester by Corollary \ref{lemfinalsylvester}. 

Now, using proof by contradiction, suppose that  
\[
\liminf_{n\to\infty} a_n^{\frac{1}{2^n}} > \lim_{n\to\infty} c_n^{\frac{1}{2^n}}.
\]
Then the set \(\{n \in \mathbb{Z}_{> 0} \mid a_n \leq c_n \}\) is finite, implying that there exists a positive integer \( N_2 \geq N_1 + n_0 + 1\) such that for all \( n > N_2 \), we have \( a_n > c_n \). Hence, we obtain  
\begin{equation}\label{ineq3}
\sum_{i=N_2+1}^{\infty} \frac{1}{a_i} < \sum_{i=N_2+1}^{\infty} \frac{1}{c_i}.
\end{equation} Furthermore, let \(n=N_2 - N_1 \) in (\ref{ineq4}), we know that  
\begin{equation}\label{ineq5}
\frac{1}{a_{N_1+1}} + \frac{1}{a_{N_1+2}} + \cdots + \frac{1}{a_{N_2}}    \leq    \frac{1}{c_{N_1+1}} + \frac{1}{c_{N_1+2}} + \cdots + \frac{1}{c_{N_2}}.  
\end{equation} Combining (\ref{ineq5}) with \(c_j = a_j\) for \( 1 \leq j \leq N_1\), it follows that 
\begin{equation}\label{ineq6}
\sum_{i=1}^{N_2} \frac{1}{a_i} \leq \sum_{i=1}^{N_2} \frac{1}{c_i}.
\end{equation} Combining (\ref{ineq3}) and (\ref{ineq6}), we conclude that  
\[
\sum_{i=1}^{\infty} \frac{1}{a_i} = \sum_{i=1}^{N_2} \frac{1}{a_i} + \sum_{i=N_2+1}^{\infty} \frac{1}{a_i}
< \sum_{i=1}^{N_2} \frac{1}{c_i} + \sum_{i=N_2+1}^{\infty} \frac{1}{c_i}
= \sum_{i=1}^{\infty} \frac{1}{c_i} = \lambda.
\]
However, this contradicts the assumption that \(\sum_{i=1}^{\infty} \frac{1}{a_i} = \lambda.\) Thus, we have completed the proof that \[
\liminf_{n\to\infty} a_n^{\frac{1}{2^n}} \leq \lim_{n\to\infty} c_n^{\frac{1}{2^n}}.
\]
\end{proof}

\subsection{Proof of Theorem \ref{mainthm2}}
By Theorem \ref{mainthm3}, there exists an eventually Sylvester sequence \(\{c_n\}_{n=1}^{\infty}\) of positive integers satisfying \(\sum_{i=1}^\infty \frac{1}{c_i} = \lambda\), such that \( \{c_n\}_{n=1}^{\infty} \) and \( \{u_n\}_{n=1}^{\infty} \) are distinct as sets, and  
\[
\liminf_{n\to\infty} a_n^{\frac{1}{2^n}} \leq \lim_{n\to\infty} c_n^{\frac{1}{2^n}}.
\] Furthermore, by Theorem \ref{mainthm1}, we have  
\[
\liminf_{n\to\infty} c_n^{\frac{1}{2^n}} = \lim_{n \to \infty} c_n^{\frac{1}{2^n}} < \lim_{n \to \infty} u_n^{\frac{1}{2^n}}.
\]
Thus, we conclude that 
\[
\liminf_{n\to\infty} a_n^{\frac{1}{2^n}} < \lim_{n\to\infty} u_n^{\frac{1}{2^n}}.
\]
This completes the proof.\hfill $\square$

\section{Concluding Remarks and Open Problems}\label{sec:conclusion}

While our constructive method helps resolve Conjecture~\ref{conj1}, addressing Theorem~\ref{mainthm2} without assuming the validity of Conjecture~\ref{conj2} remains out of reach. Solving Theorem~\ref{mainthm2} using a constructive approach similar to that in Section~\ref{sectionconstruction} appears to be nearly impossible. We regard this as a natural generalization of Conjecture~\ref{conj1}, and formally state it as the following conjecture:

\begin{conj}\label{conj3}
Let \( 0 < \frac{p}{q} \leq 1 \) be an irreducible fraction whose best \( m \)-term Egyptian underapproximation is unique for every positive integer \( m \). Let \( \{u_n\}_{n=1}^{\infty} \) denote the best Egyptian underapproximation of \( \frac{p}{q} \), and let \( a_1 < a_2 < \cdots \) be any other sequence of positive integers satisfying 
\[
\sum_{i=1}^\infty \frac{1}{a_i} = \frac{p}{q}.
\]  
Then,
\[
\liminf_{n\to\infty} a_n^{\frac{1}{2^n}} < \lim_{n\to\infty} u_n^{\frac{1}{2^n}}.
\]   
\end{conj}

On the other hand, several special cases of Conjecture~\ref{conj3} are relatively easy to resolve. For instance, when \( \frac{p}{q} \) satisfies either condition~(1) or (2) in Corollary~\ref{corpq}, Conjecture~\ref{conj3} can be established using a constructive method similar to that in Section~\ref{sectionconstruction}, by applying the results of Nathanson \cite[Theorem 1]{Nat23} and Chu \cite[Proposition 1.10]{Chu23}. However, as the construction is rather lengthy, we omit the detailed proof in this paper.

Characterizing such values of \( \frac{p}{q} \) in Conjecture~\ref{conj3} appears to be the central difficulty. This motivates the following natural problem:

\begin{problem}\label{problem1}
Characterize all rational numbers \( 0 < \frac{p}{q} \leq 1 \) for which the best \( m \)-term Egyptian underapproximation is unique for every positive integer \( m \).
\end{problem}

We remark that Nathanson \cite{Nat23} and Chu \cite{Chu23} only provided some sufficient conditions for Problem~\ref{problem1}, under which the best \( m \)-term Egyptian underapproximation also coincides with the greedy \( m \)-term Egyptian underapproximation.

Finally, confirming the validity of Conjecture~\ref{conj2} would also provide a promising route toward resolving Conjecture~\ref{conj3}:

\begin{problem}\label{problem2}
Prove or disprove Conjecture~\ref{conj2}.
\end{problem}


\begin{thebibliography}{99}

\bibitem{Badea1993} C. Badea, 
A theorem on irrationality of infinite series and applications, \textit{Acta Arithmetica}, 63(4):313–323, 1993.

\bibitem{Badea1995} C. Badea, 
On some criteria of irrationality for series of positive rationals: A survey, \textit{Actes de rencontres Arithm\'{e}tiques de Caen}, 2–15, 1995.

\bibitem{EP} T. F. Bloom, 
Erd\H{o}s problems, Available at: \url{https://www.erdosproblems.com/}, Accessed: March 15, 2025.

\bibitem{BE22} T. F. Bloom and C. Elsholtz, 
Egyptian fractions, \textit{Nieuw Arch. Wiskd. (5)}, 23(4):237–245, 2022.

\bibitem{chentoufsylvester} A. A. Chentouf, 
On Sylvester’s sequence and some of its properties, \textit{Parabola}, 56(2), 2020.

\bibitem{Chu23} H. V. Chu, 
A threshold for the best two-term underapproximation by Egyptian fractions, 
\textit{Indag. Math. (N.S.)}, 35(2):350–375, 2024.

\bibitem{chun2011egyptian} J. H. Chun, 
Egyptian fractions, Sylvester’s sequence, and the Erd\H{o}s-Straus conjecture, \textit{History}, 2(1):6, 2011.

\bibitem{EG80} P. Erd\H{o}s and R. L. Graham, 
\textit{Old and new problems and results in combinatorial number theory}, Vol. 28 of \textit{Monographies de L’Enseignement Math\'{e}matique}, Universit\'{e} de Gen\`{e}ve, L’Enseignement Math\'{e}matique, 1980.

\bibitem{Finch2003} S. R. Finch, 
\textit{Mathematical Constants}, Cambridge University Press, 2003.

\bibitem{Gra13} R. L. Graham, 
Paul Erd\H{o}s and Egyptian fractions, In \textit{Erd\H{o}s Centennial}, Vol. 25 of \textit{Bolyai Soc. Math. Stud.}, pages 289–309, 
J\'{a}nos Bolyai Math. Soc., Budapest, 2013.

\bibitem{Kamio2025} Y. Kamio, 
Asymptotic analysis of infinite decompositions of a unit fraction into unit fractions, \textit{arXiv preprint}, arXiv:2503.02317, 2025.

\bibitem{Kovac2025} V. Kova\v{c}, 
On eventually greedy best underapproximations by Egyptian fractions, \textit{Journal of Number Theory}, 268:39–48, 2025.

\bibitem{Nat23} M. B. Nathanson, 
Underapproximation by Egyptian fractions, \textit{J. Number Theory}, 242:208–234, 2023.

\bibitem{OEIS} N. J. A. Sloane et al., 
On-line encyclopedia of integer sequences, Available at \url{https://oeis.org}.

\bibitem{Sou05} K. Soundararajan, 
Approximating 1 from below using \( n \) Egyptian fractions, \textit{arXiv preprint}, arXiv:math/0502247, 2005.


\bibitem{Sylvester1880} J. J. Sylvester, 
On a point in the theory of vulgar fractions, \textit{Am. J. Math.}, 3(4):332–335, 1880.

\bibitem{Wagon1999} S. Wagon and S. Wagon, 
\textit{Mathematica in Action}, 2nd edition, Springer Science \& Business Media, 1999.

\end{thebibliography}
\end{document}